\documentclass{amsart} \usepackage{latexsym,amsxtra,amscd,ifthen} \usepackage{amsfonts} \usepackage{verbatim}
\usepackage{amsmath} \usepackage{amsthm} \usepackage{amssymb}

\input xy \xyoption{matrix} \xyoption{arrow}\xyoption{frame} 
 
 \newcommand{\edge}{\ar@{-}}

\numberwithin{equation}{section}

\theoremstyle{plain} \newtheorem{theorem}{Theorem}[section] 
\newtheorem{proposition}[theorem]{Proposition}

\theoremstyle{definition} \newtheorem{definition}[theorem]{Definition} \newtheorem{example}[theorem]{Example}

  \newtheorem*{remark*}{Remark}

\newcommand{\gnoc}{\mathrel{{\lower.2ex\hbox{$\backsim$}}\llap{\raise.45ex\hbox{=}}}}

\begin{document}

\title[] {Hopf algebras under finiteness conditions}

\author[K.A. Brown]{Kenneth A. Brown and Paul Gilmartin} \address{School of Mathematics and Statistics\\ University of Glasgow\\ Glasgow G12 8QW\\
Scotland.} \email{Ken.Brown@glasgow.ac.uk}\email{p.gilmartin.1@research.gla.ac.uk}

%\address{}
%\email{}

\maketitle

\maketitle

\begin{abstract}
This is a brief survey of some recent developments in the study of infinite dimensional Hopf algebras which are either noetherian or have finite Gelfand-Kirillov dimension. A number of open questions are listed.
\end{abstract}

\begin{center}{{\it Dedicated with thanks and appreciation to John Clark and Patrick Smith}}\end{center}

\section{Introduction}

This article\footnote{The second author's research was supported by a grant from The Carnegie Trust for the Universities of Scotland.} is a survey of recent progress in the study of infinite dimensional Hopf algebras satisfying one or both of two finiteness conditions, namely the finiteness of Gelfand-Kirillov dimension, or the noetherian condition, that is the ascending chain condition on one-sided ideals. This paper is in some sense a continuation and an updating of the earlier surveys \cite{B} and \cite{G}. In view of the volume of recent work in this area, we have had to be selective in the topics discussed. To be specific first about what is \emph{not} covered: there is nothing on the (important and active) homological aspects, including the (twisted) Calabi-Yau property and calculation of (co)homology; we mention only very briefly in \S \ref{SecB2} recent work on the classification of Hopf algebras of small GK-dimension; and we treat only in passing  in Sections \ref{SecD} and \ref{SecE} developments in the important programme \cite{AS} to classify certain pointed Hopf algebras. For a still reasonably current account of the first two omissions, see \cite{G}; work on the third topic above has recently focussed on the special case of finite dimensional pointed Hopf algebras, and for these, the excellent survey \cite{A} has recently appeared.

The topics which \emph{are} addressed here are as follows. The noetherian property is studied in \S\ref{SecA}, and Gelfand-Kirillov dimension in \S \ref{GK}. Relations of these conditions with each other, and with finite generation of the algebra, are considered, as well as some discussion on the prime and primitive spectra of Hopf algebras satisfying finiteness conditions. In the second half of the paper we specialise to the classes of pointed and connected Hopf algebras. After a brief review of some necessary terminology in \S \ref{SecZ}, we look briefly at pointed Hopf algebras in \S \ref{SecD}; then, in more detail in \S \ref{SecE}, we consider the class of connected Hopf algebras of finite Gelfand-Kirillov dimension over an algebraically closed field of characteristic $0$. As we explain in \S \ref{SecE}, this latter class of algebras can be viewed, ring-theoretically, as generalisations of enveloping algebras of Lie algebras; geometrically, as deformations of finite dimensional affine space; and, group-theoretically, as generalised unipotent groups.

A number of open questions are listed throughout the paper. The notation we use is standard - it and unexplained terminology can be found in \cite{Mo}, for example. Thus, for a Hopf algebra $H$ defined over the field $k$, the coproduct will be denoted by $\Delta$, the counit by $\epsilon$ and the antipode by $S$. $H$ is \emph{cocommutative} if $\tau\circ\Delta = \Delta$, where $\tau$ denotes the flip, $\tau(a \otimes b) = b \otimes a.$ We assume throughout that $S$ is bijective; by a result of Skryabin \cite{S} this is always the case when $H$ is semiprime noetherian, and - conjecturally - $S$ is bijective for \emph{all} noetherian Hopf algebras. The set $\{g\in H:\Delta(g)=g\otimes g\}$ of $\emph{group-like elements}$ of $H$ is denoted by $G(H)$. For $g,h\in G(H)$, we write $P_{g,h}(H)$ for the space $\{x\in{H}:\Delta(x)=x\otimes g+h\otimes x\}$ of $(g,h)$-\emph{skew-primitive elements of} $H$; then $P_{1,1}(H)$, abbreviated to $P(H)$, is the space of $\emph{primitive elements}$ of $H$.

\section{Noetherian Hopf algebras}\label{SecA}
The problem of characterising in any meaningful alternative way the class of all noetherian Hopf algebras seems well out of reach at the present time. But, for commutative or cocommutative Hopf algebras, there are the following results.

\begin{theorem}($\mathrm{Molnar}$, \cite{Mol})\label{Molnar}) {\rm (i)} A commutative Hopf algebra is noetherian if and only if it is an affine $k$-algebra.

{\rm(ii)} A cocommutative noetherian Hopf algebra is affine.

\end{theorem}

We discuss first part $\rm(ii)$. Its converse is false - consider, for example, the group algebra $kF$ of any free group $F$ of finite rank greater than one. Indeed, the following question remains open:
\medskip

\noindent{\bf Question A:} For which groups $G$ is the group algebra $kG$ noetherian?

\medskip

Generalising Hilbert's Basis Theorem, Philip Hall proved \cite[Corollary 10.2.8]{Pa} that $kG$ is noetherian when $G$ is polycyclic-by-finite. Conversely, it is easy to see that if $kG$ is noetherian then $G$ satisfies Max, the ascending chain condition on subgroups. However, we have:

\begin{theorem}($\mathrm{Ivanov}$, \cite{I})
There exist groups $G$ satisfying Max with kG not noetherian.

\end{theorem}

In view of Ivanov's result, a more tractable approach to Question A might be to ask:

\medskip

\noindent{\bf Question B:} Is there a field $k$ and group $G$ which is not polycyclic-by-finite, but for which $kG$ is noetherian?

\medskip

In characteristic $0$ the cocommutative Hopf algebras are built from group algerbas and from enveloping algebras of Lie algebras, by famous results of Cartier, Gabriel and Kostant \cite[Corollary 5.6.4(3) and Theorem 5.6.5]{Mo}. But the story regarding noetherianity is also unclear for enveloping algebras. Thus, by the proof of the Poincar$\acute{e}$-Birkhoff-Witt theorem, $\mathcal{U}(\mathfrak{g})$ is noetherian when the Lie algebra $\mathfrak{g}$ is finite dimensional over $k$. But the converse remains open, and it was only in 2013 that the following highly non-trivial result was proved:

\begin{theorem}($\mathrm{Sierra,\; Walton}$, \cite{SW})
Let $\mathfrak{g}$ be the Witt Lie algebra over a field $k$ of characteristic $0$, $$ \mathfrak{g}=\bigoplus_{n\in\mathbb{Z}}ke_{n}, \quad [e_{i},e_{j}]=(j-i)e_{i+j}.$$ Then the enveloping algebra $\mathcal{U}(\mathfrak{g})$ is not noetherian.

\end{theorem}

Emboldened by this result we might guess, as conjecture by Sierra and Walton \cite[Conjecture 0.1]{SW}, that the Lie companion to Question B has a negative answer:

\medskip

\noindent{\bf Question C:} Is there an infinite dimensional Lie algebra $\mathfrak{g}$ for which $\mathcal{U}(\mathfrak{g})$ is noetherian?

\medskip

Let's briefly consider part $(\textrm{i})$ of Theorem \ref{Molnar}. It's possible, so far as we are aware, that one direction is valid in complete generality:
\medskip

\noindent{\bf Question D} ($\textrm{Wu, Zhang}$, \cite{WZ}): Is every noetherian Hopf $k$-algebra an affine $k$-algebra?

\medskip

Question D appears to be open even for Hopf algebras which are close to being commutative, that is, for one implication of the following, part of which was asked already in \cite{B}.
\medskip

\noindent{\bf Question E:} If a Hopf algebra satisfies a polynomial identity, is it noetherian if and only if it as affine?

\medskip

\subsection{Artinian Hopf algebras}

Here, the situation is clear, thanks to a lovely result which vastly generalises a 1963 theorem for group algebras due to I.G. Connell  \cite[Theorem 10.1.1]{Pa}.

\begin{theorem}($\mathrm{Liu, Zhang}$, \cite{LZ})
A Hopf algebra is Artinian if and only if it is finite dimensional.
\end{theorem}

\section{Finite Gelfand-Kirillov dimension}\label{GK}

\subsection{}\label{addon} Let $A=k\langle V \rangle$ be a $k$-algebra generated by the $k$-subspace $V$, with $1\in V$. Recall the definition of the Gelfand-Kirillov dimension of $A$,
$$\operatorname{GKdim}A =\overline{\lim} \ \frac{\operatorname{log}(\operatorname{dim}_{k}(V^{n}))}{\operatorname{log}(n)} $$

$$\quad \quad =\operatorname{inf}\{\rho\in\mathbb{R}:\operatorname{dim}_{k}(V^{n})\leq n^{\rho} \  \forall n>>0\}.$$
The standard reference is \cite{KL}. Examples of Borho, Kraft and Warfield \cite[Theorem 2.9]{KL}, together with the Bergman Gap Theorem \cite[Theorem 2.5]{KL} show that $\operatorname{GKdim}A$ can take any value from the set $$\{0,1\}\cup [2,\infty].$$ Nevertheless, the GK-dimension of every  known Hopf algebra is either infinity or a non-negative integer. This, together with Theorems \ref{addoff} and \ref{E3}, has led to the following natural question.
\medskip

\noindent{\bf Question F}($\textrm{Zhuang}$, \cite{Zhu}): If $H$ is a Hopf algebra is $\operatorname{GKdim}H$ in $\mathbb{Z}_{\geq{0}}\cup\{\infty\}?$

\medskip

The task of classifying (in any meaningful sense) all Hopf algebras of finite GK-dimension is clearly hopeless. But important subclasses can certainly be dealt with. Let us temporarily assume in the rest of this paragraph that $k$ has characteristic $0$ and is algebraically closed (although the second hypothesis is mainly a matter of convenience). If $H$ is any affine $\emph{commutative}$ Hopf $k$-algebra then $H$ is the coordinate ring $\mathcal{O}(G)$ of an affine algebraic group $G$ over $k$, and conversely; this is a well-known equivalence of categories, see for example \cite{Wa}. Then $\operatorname{GK dim}H=\operatorname{dim}G<\infty$ \cite{KL}.

For affine \emph{cocommutative} Hopf algebras $H$ the road is rougher, but we can call on the Cartier-Gabriel-Kostant theorem \cite[Theorem 5.6.4, 5.6.5, remark on p.76]{Mo}, which presents $H$ as a smash product
 \begin{equation}\label{alphastar}H \cong \mathcal{U}(P(H))\star kG(H)\end{equation}

\noindent of the enveloping algebra of the Lie algebra $P(H)$ of primitive elements of $H$ by the group algebra of its group-like elements $G(H)$. From this and work of Zhuang we can deduce:

\begin{theorem}\label{addoff}
Let $H$ be an affine cocommutative Hopf algebra over the algebraically closed field $k$ of characteristic $0$. Then $H$ has finite GK-dimension if and only if $\operatorname{dim}_{k}(P(H))<\infty$ and $G(H)$ is finitely generated, with a nilpotent subgroup of finite index. In this case $\mathrm{GKdim} H = \mathrm{dim}_k P(H) + \mathrm{growth} G(H) \in \mathbb{Z}_{\geq 0}.$
\end{theorem}
\begin{proof}
Suppose $H$ is affine cocommutative with $\operatorname{GKdim}H<\infty$. Then the subalgebras $\mathcal{U}(P(H))$ and $kG(H)$ of $H$ occurring in ($\ref{alphastar}$) also have finite GK-dimension, so the stated conclusions on $P(H)$ and $G(H)$ follow respectively from \cite[Lemma 6.5 and Proposition 6.6]{KL} and from Gromov's theorem \cite[Theorem 11.1]{KL}. The converse follows from a special case of a theorem of Zhuang \cite[Theorem 5.4]{Zhu}; see Theorem \ref{E3} below.
\end{proof}

In particular we note that - in both the commutative and the cocomutative cases - noetherianity of $H$ is a consequence of finite GK dimension. (Noetherianity for cocommutative $H$ follows from Theorem \ref{addoff} and standard noncommutative variants of the Hilbert Basis Theorem.) This suggests:
\medskip

\noindent{\bf Question G:} Is every affine Hopf $k$-algebra of finite GK-dimension noetherian?

\medskip

The converse of Question G is easily seen to be false: take $H$ to be the group algebra of any polycylic group which is not nilpotent-by-finite. Then $H$ is noetherian, but $\operatorname{GKdim}H$ is infinite by \cite[Theorem 11.1]{KL}. Since every affine PI-algebra has finite GK-dimension by a theorem of Berele \cite[Corollary 10.7]{KL}, a positive answer to Question G would confirm one implication in Question E.

\subsection{}\label{SecB2}
When $k$ is algebraically closed of characteristic $0$, considerable progress has been made towards classifying all prime (say) Hopf algebras of ``small" GK-dimension. Here, ``small" means ``at most 2", and one imposes the prime hypothesis (or even the stronger requirement that the algebra is a domain) in order to avoid the requirement to classify all finite dimensional Hopf $k$-algebras as a subsidiary task within the classification programme. We don't have space to review this work - see \cite{BZ1}, \cite{GZ}, \cite{WZZ1} for details of the current state of play.

\subsection{}\label{SecB3}
The investigation of the structure and representation theory of the full class of noetherian Hopf algebras of finite GK-dimension is in its infancy. Of course, some very important subclasses have been intensively studied over the past 60 years - enveloping algebras of finite dimensional Lie algebras, group algebras of finitely generated nilpotent groups, quantised enveloping algebras and quantised function algebras; but little is known in general.

All the known examples have good homological properties, but in the absence of significant progress in this direction (known to us at the time of writing) since the summary in \cite[\S 6]{BZ}, we won't discuss that further here.

Regarding representation theory, following the philosophy proposed by Dixmier for enveloping algebras in the 1960s, one should start by trying to understand the primitive and prime spectra. The first significant step in this direction has recently been taken by Bell and Leung, incorporating earlier work on enveloping algebras and group algebras \cite{D}, \cite{Moe}, \cite{Z}.

\begin{theorem}($\mathrm{Bell,\; Leung}$, \cite{BL})\label{BL}
Let $H$ be an affine cocommutative Hopf algebra of finite GK-dimension over the algebraically closed field $k$ of characteristic $0$. Let $P$ be a prime ideal of $H$. Then $P$ is primitive if and only if $P$ is rational if and only if $P$ is locally closed in $\operatorname{Spec}(H)$.
\end{theorem}

A similar conclusion had earlier been obtained for quantised coordinate rings, \cite{GL}. Here, to say that $P$ is $\emph{rational}$ means that the centre of the artinian quotient ring of $H/P$ is just $k$; and $P$ is $\emph{locally closed}$ if
\[
P\subsetneq H\cap \bigcap\{Q:Q\in\operatorname{Spec}(H), P\subsetneq Q\}.
\]
When the three subclasses of prime ideals of an algebra $R$ coincide as in the theorem, so that primitivity is characterised for the ideals of $R$ by both an intrinsic algebraic property and by a topological property, we say that $R$ \emph{satisfies the Dixmier-Moeglin equivalence}. Note that Bell and Leung included the hypothesis ``H is noetherian" in \cite{BL}, but in fact this is a consequence of the other hypotheses, by Theorem \ref{addoff} and the remark following its proof. On the other hand, if the hypotheses of Theorem \ref{BL}  are weakened by changing ``of finite GK-dimension" to ``noetherian", then the statement is false - it fails for group algebras \cite{L}.This is perhaps further evidence in support of the suggestion implicit in Question G: namely, that for Hopf algebras in charactersitic 0, finiteness of GK-dimension may be a stronger and perhaps more useful working hypothesis than noetherianity.

Bell and Leung conjecture in \cite{BL} that Theorem \ref{BL} remains true with the word ``cocommutative" deleted (but now of course adding ``noetherian", since Theorem \ref{addoff} no longer applies. As the natural first step in this direction, we propose the following perhaps quite easy question. (For the definition of ``pointed", see subsection \ref{C2}.)
\medskip

\noindent{\bf Question H:} Does the Dixmier-Moeglin equivalence hold for pointed affine noetherian Hopf algebras of finite GK-dimension over an algebraically closed characteristic $0$ base field?

\medskip

\section{The coradical filtration; pointed and connected Hopf algebras}\label{SecZ}

We recall some standard concepts and notation; details can be found in \cite[Chapter 5]{Mo}.

\subsection{}\label{C1}

The \emph{coradical} $C_{0}$ of a coalgebra $C$ is the sum of the simple subcoalgebras of $C$. The coradical is the first term of the \emph{coradical filtration}, defined inductively for $i\geq 0$ by
$$
C_{i+1}=\{c\in C:\Delta(c)\in C_{i}\otimes C+C\otimes C_{0}\}.$$ This is an ascending chain and exhaustive filtration of $C$: $C_{i}\subseteq C_{i+1}$ and $\bigcup_{i}C_{i}=C$. Moreover it is a coalgebra filtration, meaning that for all $i\geq{0}$, $$\Delta(C_{i})\subseteq \sum_{0\leq j\leq i} C_{j}\otimes C_{i-j}.$$

\subsection{}\label{C2}

Clearly, the span $kG(C)$ of the group-like elements is contained in $C_{0}$. We say that $C$ is \emph{pointed} if $kG(C)=C_{0}$, equivalently if every simple subcoalgebra of $C$ is one-dimensional; and $C$ is \emph{connected} if $C_{0}=k$, equivalently if $C$ is pointed with $G(C)=\{1\}$.

\subsection{}\label{C3}

Suppose now that $H$ is a Hopf algebra. In general, the coradical filtration $\{H_{n}\}$ of $H$ is $\emph{not}$ an algebra filtration, but $\{H_{n}\}$ \emph{is} an algebra filtration when $H_{0}$ is a Hopf subalgebra of $H$. In particular, this is the case when $H$ is pointed. When $\{H_{n}\}$ is an algebra filtration we can form the associated graded algebra of $H$ with respect to its coradical filtration, $$\operatorname{gr}H:=\bigoplus_{i\geq{0}} H_{i}/H_{i-1}=\bigoplus_{i\geq{0}}H(i); \;\; H_{-1}=\{0\}.$$ There is in this case an obvious induced Hopf algebra structure on $\operatorname{gr}H$. Indeed, $\operatorname{gr}H$ is a \emph{coradically graded} Hopf algebra, meaning that, for all $n\geq{0}$ $$(\operatorname{gr}H)_{n}=\bigoplus_{i=0}^{n}H(i).$$ In particular, $\operatorname{gr}H$ is pointed [resp. connected] if $H$ is pointed [resp. connected]. Moreover, $\operatorname{gr}H$ is a \emph{graded coalgebra}, meaning that $$\Delta(H(n))\subseteq \sum_{0\leq i\leq n}H(i)\otimes H(n-i)$$ for all $n\geq{0}$.

\section{Pointed Hopf algebras of finite GK-dimension}\label{SecD}

\subsection{}\label{grounded} Suppose that $H$ is a pointed Hopf algebra. By \cite{AS} the associated graded algebra $\operatorname{gr}H$ with respect to the coradical filtration of $H$ exists, and is a graded Hopf algebra. There is an obvious surjective Hopf algebra morphism $$\pi:\operatorname{gr}H\rightarrow H(0)=kG(H).$$ Setting $R$ to be the algebra of \emph{coinvariants} $$R:=\{h\in H:(\operatorname{id}\otimes \pi)\circ \Delta(h)=h\otimes 1\},$$ one finds that $\operatorname{gr}H$ decomposes as a smash product or bosonisation,
\begin{equation}\label{boson}
\operatorname{gr}H=R \# kG(H).
\end{equation}
Here, $R$ is not in general a Hopf subalgebra; but it is a braided Hopf algebra in the category of Yetter-Drinfeld modules over $kG(H)$; see, for example, \cite{AS}. Moreover, $R$ inherits the grading from $\operatorname{gr}H$, with $R(0)=k$ and $$(\operatorname{gr}H)_{1}=(R(0)\oplus R(1))G(H).$$

\medskip

Around the end of the last century, Andruskiewitsch and Schneider began a programme to study pointed Hopf algebras by means of the above machinery. They focused on the case where the subalgebra $R$ in (\ref{boson}) is generated in degree 1 - that is,
\begin{equation}\label{gamma}
R=k\langle R(1) \rangle;
\end{equation}
when this happens $R$ is called a \emph{Nichols algebra}. They conjectured that this is always the case when $H$ is finite dimensional and $k$ is algebraically closed of characteristic 0 \cite{AS2}. A review of progress on this project up to March 2014, with many references, can be found in \cite{A}.

\medskip

For a pointed Hopf algebra $H$, the decomposition (\ref{boson}) affords a window on the GK-dimension of $H$:

\begin{theorem}($\mathrm{Zhuang}$, \cite[Corollary 3.6, Proposition 3.8 and Theorem 5.4]{Zhu})\label{zhuang}
Let $H$ be a pointed Hopf algebra.

{\rm{(i)}} $\operatorname{GKdim}H=\operatorname{sup}\{\operatorname{GKdim}A: \textrm{$A$ affine Hopf subalgebra of $H$}\}$; moreover, $G(H)$ is finitely generated if $H$ is affine.

{\rm{(ii)}} Retain the notation of subsection \ref{grounded}, (though without assuming (\ref{gamma})). Suppose that $\operatorname{dim}_{k}R(1)<\infty$ and that $\operatorname{gr}H$ is affine. Then $$\operatorname{GKdim}R+\operatorname{GKdim}kG(H)=\operatorname{GKdim}\operatorname{gr}H=\operatorname{GKdim}H.$$
\end{theorem}

In the light of the theorem and bearing in mind Question F, it's natural to ask:
\medskip

\noindent{\bf Question I:} With the above notation, is $\mathrm{GKdim}R \in \mathbb{Z} \cup \{\infty\}$?
\medskip

The proof of Theorem \ref{zhuang} is surprisingly delicate, with \rm{(i)} and the second inequality in (ii) using Takeuchi's construction \cite{T} of free Hopf algebras. It follows easily from \cite[Theorem 5.4.1 (1)]{Mo} that $\operatorname{dim}_{k}R(1)<\infty$ if and only if the space $P_{G}'(H)$ of non-trivial skew-primitive elements of $H$ is finite dimensional . Both this hypothesis and the requirement that $\operatorname{gr}H$ is affine appear rather inconvenient, but unfortunately it is not enough to simply assume that $H$ is affine, as Zhuang notes in \cite[Example 5.7]{Zhu}:

\begin{example}\label{positive}
Let $k$ be the field of $p$ elements and take $H=k[x]$, with $x$ primitive. Then $$\operatorname{gr}H\cong k[x_{1},x_{2},\ldots]/\langle x_{1}^{p},x_{2}^{p},\ldots \rangle,$$ so $$0=\operatorname{GKdim}\operatorname{gr}H<\operatorname{GKdim}H=1.$$
\end{example}

Nevertheless there is some evidence that these pathologies disappear in characteristic $0$:
\medskip

\noindent{\bf Question J:}(Wang, Zhang, Zhuang, \cite{WZZ})
Let $H$ be a pointed Hopf $k$-algebra. If $k$ has characteristic $0$ and $H$ is affine of finite GK-dimension, is $\operatorname{dim}_{k}R(1)<\infty$ and $\operatorname{gr}H$ affine?
\medskip

\section{Connected Hopf Algebras of finite Gelfand-Kirillov dimension}\label{SecE}
In this section we specialise the discussion from \S\ref{SecD} to the case where the Hopf algebra $H$ is connected, so $H_{0}=k$ and $G(H)=1$. We assume throughout this section that $k$ is algebraically closed of characteristic $0$.

\subsection{}\label{E1}
Suppose that $k$ and $H$ are as above, and $H$ is in addition \emph{cocommutative}.  Then from the isomorphism (\ref{alphastar}) in subsection \ref{addon} we see that
\medskip
\begin{center}
\textit{$H$ is the enveloping algebra $\mathcal{U}(P(H)$ of its Lie algebra $P(H)$ of primitive elements.}
\end{center}
\medskip

And, conversely, every enveloping algebra is connected. One easily shows from the PBW theorem (or deduces as a special case of Theorem \ref{lazard} below) that
$$ \operatorname{GKdim}H<\infty \iff \operatorname{dim}_{k}(P(H))<\infty, $$
and in this case these two integers are equal.

\subsection{}\label{E2}
Now suppose instead that $k$ and $H$ are as above, with $H$ affine and \emph{commutative}. Then $H$ is the algebra of polynomial functions $\mathcal{O}(G)$ of some affine algebraic group $G$, as noted in subsection \ref{addon}. Then we have:

\begin{theorem}\label{lazard}
Let $k$ be an algebraically closed field of characteristic $0$, and let $H$ be an affine commutative Hopf $k$-algebra. Then $\operatorname{GKdim}H<\infty$, and the following are equivalent.

{(\rm i)}  $H$ is connected, with $\operatorname{GKdim}H=n$.

{(\rm ii)}  $H=k[x_{1},\ldots,x_{n}]$, a polynomial $k$-algebra in $n$ indeterminates.

{(\rm iii)}  $H=\mathcal{O}(G)$ , for a unipotent algebraic group $G$, with $\operatorname{dim}G=n$.

\end{theorem}

To say that $G$ is \emph{unipotent} is equivalent to requiring it to be a closed subgroup of the group of strictly upper triangular $m\times m$ matrices over $k$, for some $m$. The only deep part of the theorem is the implication $(\rm{ii})\implies(\rm{iii})$, which is a 1955 theorem of Lazard \cite{La}.

\subsection{}\label{conngr}

The previous two paragraphs show that connected \emph{cocommutative} and the connected \emph{commutative} Hopf $k$-algebras $H$ of finite GK-dimension share a striking common feature - in both cases $H$ has an associated graded algebra (with respect to its coradical filtration) which is a commutative polynomial algebra in $\operatorname{GKdim}H$ variables, furnished with a structure of a coradically graded Hopf algebra. Thus, when $H$ is connected cocommutative, it is essentially the Poincar$\acute{\mathrm{e}}$-Birkhoff-Witt theorem which tells us that $\operatorname{gr}H$ is the symmetric algebra $S(P(H))$ of $P(H)$; its Hopf structure, as in \ref{C3}, means that $\operatorname{gr}H=S(P(H))$ is the algebra of polynomial functions of the abelian group $(k,+)^{\oplus \operatorname{dim}_k(P(H))}$. These classical results are simultaneously generalised in the following beautiful result:

\begin{theorem}($\mathrm{Zhuang}$, \cite[\textrm{Theorem 6.10}]{Zhu})\label{E3} Let $H$ be a connected Hopf $k$-algebra, with $k$ algebraically closed of characteristic $0$. Then the following are equivalent.

{(\rm i)}  $\operatorname{GKdim}H<\infty$;

{(\rm ii)}  $\operatorname{GKdim}\operatorname{gr}H<\infty$;

{(\rm iii)} $\operatorname{gr}H$ is affine;

{(\rm iv)}  $\operatorname{gr}H\cong k[x_{1},\ldots, x_{n}]$, a polynomial algebra in $n$ indeterminates.

\noindent When these conditions hold, $\operatorname{GKdim}H=\operatorname{GKdim}\operatorname{gr}H=n$.

\end{theorem}

The first key point in the proof is that
\begin{equation}\label{alpha}
\textit{$H$ a connected Hopf algebra $\implies$ $\operatorname{gr}H$ is  commutative.}
\end{equation}

This is in fact valid over $\emph{any}$ field, and goes back to Sweedler \cite[\textrm{Theorem 11.2.5 a}]{Sw}. It has been reproved several times since; Zhuang uses a lemma of Andruscieuwitsch and Schneider \cite[\textrm{Lemma 5.5}]{AS2}, to deduce that the graded dual of $\operatorname{gr}H$, that is $\bigoplus_{n\geq{0}}(\operatorname{gr}(H)(n))^{*}$, is generated in degree 1 and is therefore cocommutative. Hence, $\operatorname{gr}H$ is commutative. Yet another proof, attributed to Foissy, can be found in \cite[Proposition 1.6]{AgS}.

To prove $(\rm{ii})\implies (\rm{iii})$, Zhuang shows that a connected coradically graded Hopf algebra in characteristic $0$ must be affine if its GK dimension is finite, \cite[\textrm{Lemma 6.8, 6.9}]{Zhu} - this is the dual version of the fact that unipotent groups in characteristic $0$ are built from copies of the additive group $(k, +)$ of $k$. Notice that $(\rm{ii})\implies (\rm{iii})$ is false in positive characteristic, by Example \ref{positive}.

Given the implication (\ref{alpha}), the equivalence of $(\rm{ii})$ and $(\rm{iv})$ stems from the fact that, in characteristic $0$, $\operatorname{gr}H$ is the coordinate ring of an algebraic group, and so is smooth, (see for example \cite{Wa}). Therefore, $\operatorname{gr}H$, being a commutative connected graded algebra of finite global dimension $n$, is a polynomial algebra in $n$ indeterminates \cite[\rm{III}.2.5]{NVO}. Alternatively, one can note, from subsection \ref{C3}, that $\operatorname{gr}H$ is a connected Hopf algebra, since $H$ is, and is affine by $(ii) \Rightarrow (iii).$ Now appeal to Theorem \ref{lazard}.

Finally, the implications $(\rm{iii})\implies (\rm{i}) \implies (\rm{ii})$ are standard results on GK-dimension \cite[\textrm{Lemma 6.5 and Proposition 6.6}]{KL}.

\subsection{}
From Theorem \ref{E3} we can deduce by standard methods some important properties of these connected Hopf algebras, properties which we shouldn't find surprising given the fact that the algebras are deformations of commutative polynomial algebras. Part (\rm{i}) gives a positive answer to Question G for these algebras.

\begin{proposition}(\textrm{Zhuang})\cite[\textrm{Corollary 6.11}]{Zhu}.
Let $H$ be a connected Hopf $k$-algebra with $\operatorname{GKdim}H=n<\infty$, with $k$ algebraically closed of characteristic $0$.

{(\rm i)} $H$ is a noetherian domain of Krull dimension at most $n$.

{(\rm ii)} $H$ is AS-regular and Auslander-regular, of global dimension $n$.

{(\rm iii)} $H$ is GK-Cohen-Macaulay.

\end{proposition}

Unexplained terminology used above can be found in many references, for example \cite{BZ}. The fact that $H$ is a domain does not need $\operatorname{GKdim}H<\infty$, and is attributed by Zhuang to Lebruyn. It's natural to ask whether noetherianity is \emph{equivalent} to finite GK-dimension for these connected Hopf algebras:

\medskip

\noindent{\bf Question K:} Let $k$ be algebraically closed of characteristic $0$, and let $H$ be a connected Hopf $k$-algebra. If $H$ is noetherian, is $\operatorname{GKdim}H<\infty$?

\medskip

 Suppose we could show that, for $H$ as in question K, $\operatorname{gr}H$ is also noetherian. Then $\operatorname{gr}H$ is affine by Molnar's Theorem, Theorem \ref{Molnar}(\rm{i}), and so $\operatorname{GKdim}H<\infty$ by Theorem \ref{E3}. But notice that the seemingly innocuous Question K contains as a special case the characteristic $0$ case of the notorious Question C!

\subsection{}\label{E5}

Following the work of Zhuang outlined above there has been considerable further research on connected Hopf algebras in characteristic $0$. First, all such algebras of GK dimension at most $4$ have been  classified, in \cite{Zhu} and \cite{WZZ}. Of course, by subsection \ref{E1}, we always have for each $n\geq{0}$, the enveloping algebras $H=\mathcal{U}(\mathfrak{g})$ of the Lie algebras $\mathfrak{g}$ with $\operatorname{dim}_{k}(\mathfrak{g})=n$. For $n=0,1,2,$ it is not hard to show \cite[\textrm{Proposition 7.5}]{Zhu} that this completes the list. However:

\begin{theorem}($\mathrm{Wang, Zhang, Zhuang}$, \cite{Zhu}, \cite{WZZ})\label{ThmE5} Let $H$ be a connected Hopf $k$-algebra, where $k$ is algebraically closed of characteristic $0$.

{(\rm i)} If $\operatorname{GKdim}H=3$, then $H$ is either isomorphic as a Hopf algebra to $\mathcal{U}(\mathfrak{g})$, where $\mathfrak{g}$ is a Lie algebra of dimension $3$, or $H$ is a member of one of two explicitly defined (infinite) families.

{(\rm ii)} If $\operatorname{GKdim}H=4$, then $H$ is either isomorphic as a Hopf algebra to $\mathcal{U}(\mathfrak{g})$, where $\mathfrak{g}$ is a Lie algebra of dimension $4$, or $H$ is a member of one of 12 explicitly defined families.

\end{theorem}

Note that the complex Lie algebras of dimension at most $4$ have been classified -  see e.g. \cite[\textrm{p.209, Theorem 1.1}]{OV}. To give some feel for the algebras of the theorem, here is a sample of one of the families of GK dimension 3, as in (i).

\begin{example}\label{ExaE5}
Let $\lambda\in {k}$, and let $B(\lambda)$ be the $k$-algebra generated by $x$, $y$ and $z$, subject to the relations $$ [x,y]=y, [x,z]=z-\lambda y, [y,z]=0. $$  The coalgebra structure on $B(\lambda)$ is given by letting $x$ and $y$ be primitive, with $$\Delta(z)=z\otimes 1 +1\otimes z +x\otimes y-y\otimes x$$ and $$\epsilon(x)=\epsilon(y)=\epsilon(z)=0.$$ The antipode is given by $$S(x)=-x, S(y)=-y, S(z)=-z+y.$$
\end{example}

It's worth noting that, for all $n\geq{1}$, $S^{n}(z)=(-1)^{n}(z-ny)$, so that $S$ has infinite order, in contrast to the situation for commutative or cocommutative Hopf algebras, which are always \emph{involutary}, meaning that $S^{2}=\operatorname{Id}$, \cite[\textrm{Corollary 1.5.12}]{Mo}.

It is clear from the definition above that the algebra $B(\lambda)$ is isomorphic, \emph{as an algebra}, to the enveloping algebra $\mathcal{U}(\mathfrak{g}_{\lambda})$ of a Lie algebra $\mathfrak{g}_{\lambda}$. In fact, this is the case for \emph{all} the Hopf algebras featuring in Theorem \ref{ThmE5}, as well as all the cocommutative or commutative connected Hopf algebras of subsections \ref{E1} and \ref{E2}. This makes the following question a pressing one:

\medskip

\noindent{\bf Question L:} Over an algebraically closed field $k$ of characteristic $0$, is every connected Hopf algebra of finite GK-dimension isomorphic as an algebra to the enveloping algebra of a finite dimensional Lie algebra?

\medskip

We can see no evidence in favour of a positive answer, apart from the current absence of a counterexample.

Clearly, the business of listing the residents of the zoo of connected Hopf $k$-algebras is a valuable one, but it is an enterprise doomed to failure if extended beyond small GK-dimension. It's quite surprising, in fact, that a complete catalogue has been obtained up to GK-dimension 4. An approach geared to recognising structural features of large subclasses will likely be needed to describe the range of algebras occurring in higher dimensions, One such structure is discussed in the next subsection.

\subsection{}\label{E6}
Dualizing a basic property of an affine unipotent group $G$ in characteristic $0$, namely that it has a finite chain $1=G_{0}\subset G_{1}\subset \ldots \subset G_{n}=G$ of normal subgroups with $G_{i+1}/G_{i}\cong (k,+)$ for $0\leq i\leq n$, it is natural to make the

\begin{definition}
An \emph{iterated Hopf Ore extension} (IHOE) is a Hopf $k$-algebra $H$ with a chain of Hopf subalgebras
\begin{equation}\label{ore}
k=H_{(0)}\subset H_{(1)}\subset \ldots \subset H_{(n)}=H
\end{equation}
with $H_{(i+1)}=H_{(i)}[x_{i+1};\sigma_{i+1},\delta_{i+1}]$ a skew polynomial extension for $0\leq i < n$.
\end{definition}

Here, $\sigma_{i+1}$ is an algebra automorphism of $H_{(i)}$ and $\delta_{(i+1)}$ is a $\sigma_{i+1}$-derivation, for all $i$; see, for example \cite[\S 1.2]{MR}.

IHOEs are introduced and studied in \cite{BOZZ}. Their relevance to the present discussion of connected Hopf algebras is clear from (\rm{i}) and (\rm{ii}) of:

\begin{theorem}(\cite[\textrm{Theorems 1.3, 1.5}]{BOZZ})\label{ThmE6} Let $k$ be a field and $H$ an IHOE with defining series (\ref{ore}) and antipode $S$.

{(\rm i)} $H$ is noetherian, of GK dimension $n$.

{(\rm ii)} $H$ is connected.

{(\rm iii)} Either $S^{2}=\operatorname{Id}$ or $S$ has infinite order.

{(\rm iv)} Explicit conditions can be given on $\Delta, \epsilon, S$, and on the possible choices of $\{\sigma_i, \delta_i : 2 \leq i \leq n \}.$

\end{theorem}

For details of (iv), see \cite{BOZZ}. Note that Example \ref{ExaE5} shows that both possibilities in $(\rm{iii})$  of the theorem can occur, even for the same algebra endowed with two different coalgebra structures. Many, but \emph{not} all, of the known connected Hopf algebras of finite GK-dimension of characteristic $0$ are IHOEs. For example, if $\mathfrak{g}$ is any semisimple finite dimensional complex Lie algebra which is \emph{not} a direct sum of copies of $\mathfrak{sl}(2,\mathbb{C})$, then $\mathcal{U}(\mathfrak{g})$ is connected of finite GK-dimension, by subsection \ref{E1}, but is \emph{not} an IHOE - the point is that $\mathfrak{g}$ does not have enough Lie subalgebras to allow the construction of a chain as in (\ref{ore}).

There are many open questions concerning IHOEs, for which we suggest the interested reader consult \cite{BOZZ}. Here we ask instead

\medskip

\noindent{\bf Question M:} Find another general construction of connected Hopf algebras of finite GK-dimension, different from those of subsections \ref{E1}, \ref{E2} and \ref{E6}.

\medskip

\subsection{}\label{E7}

Finally, let us briefly discuss the relation between connected Hopf algebras of \S \ref{SecE}, and the Andruskiewitsch-Schneider programme on pointed Hopf algebras which was briefly outlined in \S \ref{SecD}. The latter programme begins by considering the associated graded algebra $\operatorname{gr}H$ of a pointed Hopf algebra $H$, exactly as does Theorem \ref{E3} for connected Hopf algebras. In the (more general) pointed case we get the isomorphism (\ref{boson}), that $$\operatorname{gr}H\cong R\# G(H),$$ while in the connected case $G(H)=\{1\}$ and we thus have $\operatorname{gr}H\cong R$. However, Andruskiewitsch-Schneider impose the extra hypothesis (\ref{gamma}) on $R$, that it is \emph{generated in degree 1}, in order to be able to call on the Nichols algebra machinery. If $H$ is connected with $\operatorname{gr}H$ generated in degree $1$, then $\operatorname{gr}H$, being generated by primitive elements, is cocommutative; and one finds from Theorem \ref{E3} and its proof that $H$ itself is cocommutative, hence an enveloping algebra of a Lie algebra, as a Hopf algebra. To sum up: the intersection of hypothesis (\ref{gamma}) with subsections \ref{E1} to \ref{E6} consists precisely in the enveloping algebras of finite dimensional Lie algebras, with their standard cocommutative coproducts.


\begin{thebibliography}{10}



\bibitem{AgS} M. Aguiar and F. Sotile, Cocommutative Hopf algebras of permutations and trees, \emph{J. Alg Combinatorics} 22 (2005), 451-470.

\bibitem{A} N. Andruskiewitsch, On finite-dimensional Hopf algebras, arXiv:1403.7838v1.

\bibitem{AS} N. Andruskiewitsch and H.J Schneider, Pointed Hopf algebras, \emph{New directions in Hopf algebras}, 1–68, Math. Sci. Res. Inst. Publ., 43, Cambridge Univ. Press, 2002;
arXiv: math/ 01101136v1.

\bibitem{AS2} N. Andruskiewicz and H.J. Schneider, Finite quantum groups and quantum matrices, Adv. Math. 154 (2000),1-45.

\bibitem{B} K.A. Brown, Noetherian Hopf algebras, \emph{Turkish J. Math} Turkish J. Math. 31 (2007),  suppl., 7 – 23; arXiv:0709.2334v1


\bibitem{BL} J.P. Bell and W.H. Leung, The Dixmier-Moeglin equivalence for cocommutative Hopf algerbas of finite Gelfand-Kirillov dimension, arXiv:1403.7190v1


\bibitem{BOZZ} K.A. Brown, S. O'Hagan, J.J. Zhang and G. Zhuang, Connected Hopf algebras and iterated Ore extensions, arXiv 1308.1998.

\bibitem{BZ}  K.A. Brown and J.J. Zhang, Dualising complexes and twisted Hochschild (co)homology for noetherian Hopf algebras, \emph{J. Algebra} 320 (2008), 1814-1850.

\bibitem{BZ1} K.A. Brown and J.J. Zhang,Prime regular Hopf algebras of GK-dimension one, \emph{Proc. London Math. Soc.} 101 (2010), 260-302.


\bibitem{D} J. Dixmier, Id$\acute{\mathrm{e}}$aux primitifs dan les alg$\grave{\mathrm{e}}$bres enveloppantes, \emph{J. Algebra} 48 (1977), 96-112.

\bibitem{G} K.R. Goodearl, Noetherian Hopf algebras, \emph{Glasgow Math. J.} 55 2013, 75-87; arXiv:1201.4854v1.


\bibitem{GZ} K.R. Goodearl and J.J. Zhang, Noetherian Hopf algebra domains of Gelfand-Kirillov dimension two, \emph{J. Algebra} 324 (2010), 3131-3168

\bibitem{GL} K.R. Goodearl and E.S Letzter, The Dixmier-Moeglin equivalence in quantum coordinate rings and quantised Weyl algebras, \emph{Trans. Amer. Math. Soc.} 352 (2000), 1381-1403.

\bibitem{I} S.V. Ivanov, Group rings of Noetherian groups, , Mat. Za(SP)ntki 46 (1989), 61-66; translated in Math. Notes 46 (1989), 929-933 (1990).

\bibitem{KL} G. Krause and T.H. Lenagan, \emph{Growth of Alegbras and Gelfand-Kirillov Dimension}, (revised edition), Grad. Studies in Maths. 22, Amer. Math. Soc. 2000.

\bibitem{La} M. Lazard, Sur la nilpotence de certains groupes algebriques, \emph{C. R. Acad. Sci. Ser. 1 Math.} 41 (1955), 1687-1689.

\bibitem{LZ} C.H. Liu and J.J. Zhang, Artinian Hopf algebras are finite dimensional, \emph{Proc. Amer. Math. Soc.}, 135 (2007), 1679-1680.

\bibitem{L} M. Lorenz, Primitive ideals of group algebras of supersolvable groups, \emph{Math. Ann.} 225 (1977), 115-122.

\bibitem{MR} J.C M$^{c}$Connell and J.C. Robson, \emph{Noncommutative Noetherian Rings}, John Wiley and Sons, 1988.

\bibitem{Moe} C. Moeglin, Id$\acute{e}$aux primitifs des alg$\grave{e}$bres enveloppantes, \emph{J. Math. Pures. Appl.} 59 (1980) 265-336.


\bibitem{Mol} R.K. Molnar, A commutative noetherian Hopf algebra over a field is finitely generated, \emph{Proc. Amer. Math. Soc.} 51 (1975) 501-502.


\bibitem{Mo} S. Montgomery, \emph{Hopf Algebras and their Actions on Rings}, CBMS Regional Conference Series in Mathematics 82, Providence,RI, 1993.

\bibitem{NVO} C. Nastasescu and F. Van Ostaeyen, \emph{Graded Ring Theory}, North Holland, 1982.

\bibitem{OV} A.L. Onishchick and E.B. Vinberg (editors), Lie Groups and Lie Algebras \rm{III}, \emph{Encyclopedia of Math. Science} 41, 1994.

\bibitem{Pa} D.S. Passman, \emph{The Algebraic Structure of Group Rings}, reprint of the 1977 original, Robert E. Kreiger, Melbourne, FL, 1985.

\bibitem{SW} S.Sierra and C.Walton, The universal enveloping algebra of the Witt algebra is not noetherian, arXiv 1304.4114.


\bibitem{S} S. Skryabin, New results on the bijectivity of antipode of a Hopf algebra, \emph{J. Algebra} 306 (2006), 622-633.

\bibitem{Sw} M.E. Sweedler, \emph{Hopf Algebras}, W. A. Benjamin, New York, 1969.

\bibitem{T} M. Takeuchi, Free Hopf algebras generated by coalgebras, \emph{J. Math. Soc. Japan} 23. (1971), 561-582.

\bibitem{WZZ1} D.-G. Wang, J.J. Zhang, G. Zhuang, Hopf algebras of GK-dimension two with vanishing Ext-group. \emph{J. Algebra}  388  (2013), 219–247.

\bibitem{WZZ} D.G. Wang, J.J. Zhang and G. Zhuang, Connected Hopf algebras of Gelfand-Kirillov dimension 4, arXiv 1302.2270v1.

\bibitem{Wa} Waterhouse, \emph{Introduction to Affine Group Schemes} Graduate Texts in Mathematics 66, Springer-Verlag, New York-Berlin, 1979.

\bibitem{WZ} Q.-S. Wu and J.J. Zhang, Noetherian PI Hopf algebras are Gorenstein, \emph{Trans. Amer. Math. Soc.} 355 (2002), 1043-1066.

\bibitem{Zhu} G.Zhuang, \emph{Properties of pointed and connected Hopf algebras of finite Gelfand-Kirillov dimension}, arXiv:1202.4121v2.

\bibitem{Z} A.E. Zalesskii, The irreducible representations of finitely generated nilpotent groups without torsion, \emph{Mat. Zametki} 9 (1971), 199-210.

\end{thebibliography}
\end{document}